\documentclass[12pt]{amsart}
\usepackage{amsmath, amsthm, amsfonts, amscd, amssymb, eucal, latexsym, mathrsfs, appendix, parskip, tikz}
\usepackage[numbers,sort&compress]{natbib}
\usepackage{color,hyperref}
\hypersetup{colorlinks,breaklinks,
            linkcolor=red,urlcolor=red,
            anchorcolor=red,citecolor=blue}

\theoremstyle{definition}
\newtheorem{theorem}{Theorem}[section]
\newtheorem{corollary}[theorem]{Corollary}

\newtheorem{proposition}[theorem]{Proposition}

\newtheorem{definition}[theorem]{Definition}
\newtheorem{remark}[theorem]{Remark}

\allowdisplaybreaks

\begin{document}
\title{A generalized Fej\'er's theorem for locally compact groups}
\author{Huichi Huang}
\address{Huichi Huang, College of Mathematics and Statistics, Chongqing University, Chongqing, 401331, PR China}
\email{huanghuichi@cqu.edu.cn}
\date{May 28, 2016}
\begin{abstract}
The classical Fej\'er's theorem is a criterion for pointwise convergence of Fourier series on the unit circle. We  generalize it  to locally compact groups.
\end{abstract}

\maketitle


\section{Introduction}

Let $\mathbb{T}$ be the unit circle. When necessary, identify $\mathbb{T}$ with $[0,1)$ or $\mathbb{R}/\mathbb{Z}$. Denote $e^{2\pi inx}$ by $e(nx)$ for a real number $x$ and an integer $n$. For an $f$ in $L^1(\mathbb{T})$, its Fourier coefficients $\hat{f}(n)$ for $n\in\mathbb{Z}$ is given by
$\hat{f}(n)=\int_\mathbb{T} f(x)e(-nx)\,dx$ and define $S_N(f)(x)$ as $\sum_{n=-N}^N \hat{f}(n)e(nx)$ for every nonnegative integer $N$. The  Fourier series of $f$ is
$\sum_{n\in\mathbb{Z}} \hat{f}(n)e(nx)$.

In  classical Fourier analysis, an important question is pointwise convergence of $S_N(f)$ for  $f$ in $L^1(\mathbb{T})$.  By Carleson-Hunt theorem~\cite[Thm. (c)]{Carleson1966}~\cite[Thm. 1]{Hunt1968},  the Fourier series $S_N(f)(x)$ of an $f$ in $L^p(\mathbb{T})$ with $1<p<\infty$ converges to $f(x)$ almost everywhere. But in general the pointwise convergence does not hold even when $f$ is continuous. Cf.~\cite[Chap. II, Sec. 2]{Katznelson2004}.

If the Fourier series is replaced by its average, called the Ces\`aro mean or the Fej\'er mean,
 $$\sigma_N(f,x)=\frac{1}{N+1}[S_0(f)(x)+\cdots+S_N(f)(x)]=K_N*f(x),$$ then one have a much better convergence. Here $K_N(x)$ is the Fej\'er's kernel~(see section 2 for the definition).

In 1900, L. Fej\'er came up with an explicit criteria which tells us when and to which value $\sigma_N(f)$ converges. Cf.~\cite{Fejer1900}\cite[Chap. I, Thm. 3.1(a)]{Katznelson2004}\cite[Thm. 3.4.1]{Grafakos2014}.~\footnote{For the history of Fej\'er's theorem, see~\cite{Kahane2006}, wherein some  applications and continuations of Fej\'er's theorem are also mentioned.}


\begin{theorem}~[Fej\'er's theorem]\
~\label{Fejer1}\

For an $f$ in $L^1(\mathbb{T})$,  if both the left and the right limit of $f(x)$  exist at some $x_0$ in $\mathbb{T}$~(denoted by $f(x_0+)$ and $f(x_0-)$ respectively),   then  $$\lim_{N\to\infty}K_N*f(x_0)=\frac{1}{2}[f(x_0+)+f(x_0-)].$$
\end{theorem}

In particular, when $f$ is continuous $\sigma_N(f, x)$ converges to $f(x)$ for every $x$ in $\mathbb{T}$.

Note that the left and right limits of $f$ at $x_0$ can be interpreted in terms of the finite partition $\{(0,\frac{1}{2}),[\frac{1}{2},1)\}$ of $\mathbb{T}=[0,1)$~(up to measure 0 since $(0,\frac{1}{2})\cup[\frac{1}{2},1)=(0,1)=[0,1)\setminus\{0\}$):~\footnote {Note that $y\to 0$ when $y\in [\frac{1}{2},1)$ makes sense since we identify 0 with 1.}
$$f(x_0-)=\displaystyle\lim_{\substack{y\to 0 \\ y\in (0,\frac{1}{2})}}f(x_0-y),  \quad  f(x_0+)=\displaystyle\lim_{\substack{y\to 0 \\ y\in [\frac{1}{2},1)}}f(x_0-y).$$

Moreover for the approximate identity~(cf. Definition~\ref{def:AI}) $\{K_n\}_{n=1}^\infty$ of $L^1(\mathbb{T})$,  one have $$\int_{[0,\frac{1}{2})}K_n(t)\,dt=\int_{[\frac{1}{2},1)}K_n(t)\,dt=\frac{1}{2}$$ for all $n\geq 0$.

This observation motivates the following generalization of Fej\'er's theorem to locally compact groups.

Let $G$ be a locally compact group with the unit $e_G$ and a fixed left Haar measure $\mu$.

A finite collection $\{A_1,A_2,\cdots, A_k\}$ of Borel subsets  of $G$ is called a {\bf local partition}~(at $e_G$) if the following are true:
\begin{enumerate}
\item $A_i\cap A_j=\emptyset$ for $1\leq i\neq j\leq k$,
\item $\displaystyle\mu(G\setminus \bigcup_{i=1}^k A_i)=0$,
\item each $A_j\cap \mathcal{N}\neq\emptyset$ for any neighborhood $\mathcal{N}$ of $e_G$.
\end{enumerate}


\begin{theorem}~\label{Thm: GFejer}
[A generalized  Fej\'er's theorem]\

Consider  a locally compact group $G$ with a fixed left Haar measure $\mu$. Let $\{F_\theta\}_{\theta\in\Theta}$ be an approximate identity of $L^1(G)$.  Assume that there exists  a local partition $\{A_1,A_2,\cdots, A_k\}$  of $G$ such that  $\displaystyle\lim_{\theta}\int_{A_j}F_\theta(y)\,d\mu(y)=\lambda_j$ for every $1\leq j\leq k$.

For an $f$  in $L^\infty(G)$, if there exists  $x$ in $G$ such that $\displaystyle\lim_{\substack{y\to e_G \\ y\in A_j}} f(y^{-1}x)$~(denoted by $f(x,A_j)$) exists for every $1\leq j\leq k$, then
\begin{equation*}~\label{eqfejer2}
\lim_{\theta}F_\theta*f(x)=\sum_{j=1}^k \lambda_j f(x,A_j).
\end{equation*}
Moreover if $\displaystyle\lim_{\theta}\sup_{y\in\mathcal{N}^c}|F_\theta(y)|=0$ for any neighborhood $\mathcal{N}$ of $e_G$,  then  for every $f$  in $L^1(G)$~(or $L^\infty(G)$) such that each $f(x,A_j)$ exists for some $x$ in $G$, we have
\begin{equation*}
\lim_{\theta}F_\theta*f(x)=\sum_{j=1}^k \lambda_j f(x,A_j).
\end{equation*}
\end{theorem}

The paper is organized as follows.

In section 2, after some preliminaries, we prove Theorem~\ref{Thm: GFejer} and its  variant Corollary~\ref{cor: GFejer}. To give some applications, various special cases~(either abelian or non-abelian groups) of Theorem~\ref{Thm: GFejer} are discussed in section 3.

\section*{Acknowledgements}
Most  of the paper was finished when I was a postdoctoral supported by ERC Advanced Grant No. 267079. I would like to express my gratitude to my mentor Joachim Cuntz. Part of the paper is carried out during a visit to Besan\c{c}on, I thank Quanhua Xu for his invitation and hospitality, and Simeng Wang for his help. I  thank Loukas Grafakos for answering my question as reading his excellent GTM book: {\it Classical Fourier Analysis}. At last I  thank Stefan Cobzas and Ferenc Weisz for helpful discussions and comments.

\section{The main theorem}

Within this article $G$ stands for a locally compact group with a fixed left Haar measure $\mu$. Let $e_G$ be the identity of $G$. Denote by $L^1(G)$ the space of integrable functions~(with respect to $\mu$) on $G$ and by $L^\infty(G)$ the space of essentially bounded  functions~(with respect to $\mu$) on $G$.

The {\bf convolution} $f*g$ for $f$ and $g$ in $L^1(G)$ is given by
$$f*g(x)=\int_G f(y)g(y^{-1}x)\, d\mu(y)$$ for every $x\in G$.


\begin{definition}~[Approximate identity]~\label{def:AI}~\cite[Defn. 1.2.15.]{Grafakos2014}\

An {\bf approximate identity} is a family of functions $\{F_\theta\}_{\theta\in\Theta}$ in $L^1(G)$ such that
\begin{enumerate}
  \item $\|F_\theta\|_{L^1(G)}\leq C$ for all $\theta$.
  \item $\int_G F_\theta(x)\,d\mu(x)=1$ for all $\theta$.
  \item $\displaystyle\lim_{\theta}\int_{\mathcal{N}^c} |F_\theta(x)|\,d\mu(x)=0$ for any neighborhood $\mathcal{N}$ of $e_G$.
\end{enumerate}

\end{definition}

There always exists an approximate identity in $L^1(G)$~\cite[Chap.2, Prop. 2.42]{Folland1995}.

Now we are ready to prove the main theorem.

\begin{proof}[Proof of Theorem~\ref{Thm: GFejer}]\

Suppose $f$ is in $L^\infty(G)$ such that each $f(x,A_j)$ exists for some $x$ in $G$.

For an arbitrary $\varepsilon>0$,  there exists a neighborhood $\mathcal{N}$ of $e_G$ such that $|f(y^{-1}x)-f(x,A_j)|<\varepsilon$  for every $1\leq j\leq k$ whenever $y$ is in $\mathcal{N}\cap A_j$.

Then
\begin{align*}
&F_\theta*f(x)-\sum_{j=1}^k \lambda_j f(x,A_j)=\int_G F_\theta(y)f(y^{-1}x)\,d\mu(y)-\sum_{j=1}^k \lambda_j f(x,A_j)\\
&=(\int_{\mathcal{N}}+\int_{\mathcal{N}^c})F_\theta(y)f(y^{-1}x)\,d\mu(y)-\sum_{j=1}^k \lambda_j f(x,A_j).
\end{align*}

First
\begin{align*}
&\displaystyle\int_{\mathcal{N}} F_\theta(y)f(y^{-1}x)\,d\mu(y)-\sum_{j=1}^k \lambda_j f(x,A_j)  \\
&=\sum_{j=1}^k \int_{\mathcal{N}\cap A_j} F_\theta(y)f(y^{-1}x)\,d\mu(y)-\sum_{j=1}^k \lambda_j f(x,A_j) \\
&=\sum_{j=1}^k \int_{\mathcal{N}\cap A_j} F_\theta(y)(f(y^{-1}x)-f(x,A_j))\,d\mu(y)+\sum_{j=1}^k(\int_{\mathcal{N}\cap A_j} F_\theta(y)\,d\mu(y)-\lambda_j) f(x,A_j).
\end{align*}

So we have
\begin{align*}
&\limsup_{\theta}|\int_{\mathcal{N}} F_\theta(y)f(y^{-1}x)\,d\mu(y)-\sum_{j=1}^k \lambda_j f(x,A_j)|  \\
&\leq\sum_{j=1}^k\limsup_{\theta} \int_{\mathcal{N}\cap A_j} |F_\theta(y)||f(y^{-1}x)-f(x,A_j)|\,d\mu(y)   \\
&+\sum_{j=1}^k\limsup_{\theta}|\int_{\mathcal{N}\cap A_j} F_\theta(y)\,d\mu(y)-\lambda_j|\,| f(x,A_j)| \\
&\leq \limsup_{\theta}\sum_{j=1}^k \varepsilon\int_{\mathcal{N}\cap A_j}|F_\theta(y)|\,d\mu(y)+\sum_{j=1}^k\limsup_{\theta}|\int_{\mathcal{N}\cap A_j} F_\theta(y)\,d\mu(y)-\lambda_j|\,| f(x,A_j)|.
\end{align*}

Note that for every $\theta$
 $$\sum_{j=1}^k \int_{\mathcal{N}\cap A_j}|F_\theta(y)|\,d\mu(y)=\int_{\mathcal{N}} |F_\theta(y)|\,d\mu(y)\leq \|F_\theta\|_{L^1(G)}\leq C,$$

and it follows from
$\displaystyle\lim_{\theta}\int_{\mathcal{N}^c} |F_\theta(y)|\,d\mu(y)=0$ that
$$\lim_{\theta}\int_{\mathcal{N}\cap A_j} F_\theta(y)\,d\mu(y)=\lim_{\theta}\int_{A_j} F_\theta(y)\,d\mu(y)=\lambda_j$$ for every $1\leq j\leq k$.

Therefore
$$\limsup_{\theta}|\int_{\mathcal{N}} F_\theta(y)f(y^{-1}x)\,d\mu(y)-\sum_{j=1}^k \lambda_j f(x,A_j)|\leq C\varepsilon.$$

Moreover
$$\limsup_{\theta}|\int_{\mathcal{N}^c} F_\theta(y)f(y^{-1}x)\,d\mu(y)|\leq [\limsup_{\theta}\int_{\mathcal{N}^c} |F_\theta(x)|\,d\mu(x)] \|f\|_{L^\infty(G)}=0.$$

Hence $$\limsup_{\theta}|F_\theta*f(x)-\sum_{j=1}^k \lambda_j f(x,A_j)|\leq C\varepsilon$$ for any $\varepsilon>0$. This proves the first part of the theorem.

Now assume that $\displaystyle\lim_{\theta}\sup_{y\in\mathcal{N}^c}|F_\theta(y)|=0$ for any neighborhood $\mathcal{N}$ of $e_G$ and   $f$  is in $L^1(G)$ such that  each $f(x,A_j)$ exists for some $x$ in $G$.

As before, we have
$$\limsup_{\theta}|\int_{\mathcal{N}} F_\theta(y)f(y^{-1}x)\,d\mu(y)-\sum_{j=1}^k \lambda_j f(x,A_j)|\leq C\varepsilon$$ for every $\varepsilon>0$.

Moreover
$$\limsup_{\theta}|\int_{\mathcal{N}^c} F_\theta(y)f(y^{-1}x)\,d\mu(y)|\leq [\lim_{\theta}\sup_{y\in\mathcal{N}^c}|F_\theta(y)|]\|f\|_{L^1(G)}=0.$$
 This completes the proof.
\end{proof}

Although Theorem~\ref{Thm: GFejer} requires that a local partition $\{A_1,\cdots, A_k\}$ satisfies that every $\displaystyle\lim_{\theta}\int_{A_j}F_\theta(y)\,d\mu(y)$ exists,  this assumption could be easily satisfied when one considers subnets of $\{F_\theta\}$. Note that
$$|\int_{A_j}F_\theta(y)\,d\mu(y)|\leq \int_{A_j}|F_\theta(y)|\,d\mu(y)\leq \|F_\theta\|_{L^1(G)}\leq C$$ for all $\theta$ and $j$, so for any given  approximate identity $\{F_\theta\}_{\theta\in\Theta}$ of $L^1(G)$ and  local partition $\{A_1,\cdots, A_k\}$ of $G$, there always exists a subnet $\Theta_1$ of $\Theta$ such that every $\displaystyle\lim_{\theta\in\Theta_1}\int_{A_j}F_\theta(y)\,d\mu(y)$~(denoted by $\lambda_j(\Theta_1)$) exists.

The argument goes as follows.

Consider the net of bounded complex numbers $\{\int_{A_1}F_\theta(y)\,d\mu(y)\}_{\theta\in \Theta}$. There is a subnet $\Theta'$ of $\Theta$ such that $\displaystyle\lim_{\theta\in\Theta'}\int_{A_1}F_\theta(y)\,d\mu(y)$ exists and equals some $\lambda_1$.  Then consider the net of bounded complex numbers $\{\int_{A_2}F_\theta(y)\,d\mu(y)\}_{\theta\in \Theta'}$, as before, there exists a subnet of $\Theta''$ of $\Theta'$ such that $\displaystyle\lim_{\theta\in\Theta''}\int_{A_2}F_\theta(y)\,d\mu(y)$ exists and equals some $\lambda_2$~(also note that $\displaystyle\lim_{\theta\in\Theta''}\int_{A_1}F_\theta(y)\,d\mu(y)=\lambda_1$). Repeat this procedure. After finite steps, we can find a subnet $\Theta_1$ of $\Theta$ such that every $\displaystyle\lim_{\theta\in\Theta_1}\int_{A_j}F_\theta(y)\,d\mu(y)$ exists.

So we have  the following  variant of Theorem~\ref{Thm: GFejer}.


\begin{corollary}~\label{cor: GFejer}
Given any approximate identity $\{F_\theta\}_{\theta\in\Theta}$ of $L^1(G)$ and local partition $\{A_1,\cdots, A_k\}$ of $G$, if for an $f$ in $L^\infty(G)$, every $f(x, A_j)$ exists, then there exists a subnet $\Theta_1$ of $\Theta$ such that every $\displaystyle\lim_{\theta\in\Theta_1}\int_{A_j}F_\theta(y)\,d\mu(y)$ exists~(denoted by $\lambda_j(\Theta_1)$)  and
$$\displaystyle\lim_{\theta\in\Theta_1} F_\theta*f(x)=\sum_{j=1}^k \lambda_j(\Theta_1)f(x,A_j).$$
Moreover if $\displaystyle\lim_{\theta\in\Theta_1}\sup_{y\in\mathcal{N}^c}|F_\theta(y)|=0$ for every neighborhood $\mathcal{N}$ of $e_G$, then for every $f$ in $L^1(G)$~(or $L^\infty(G)$) such that every $f(x,A_j)$ exists for some $x$ in $G$, we have
$$\displaystyle\lim_{\theta\in\Theta_1} F_\theta*f(x)=\sum_{j=1}^k \lambda_j(\Theta_1)f(x,A_j).$$
\end{corollary}


\section{Some special cases}


In this section we consider  some concrete examples of Theorem~\ref{Thm: GFejer} including both abelian and non-abelian groups.

A local partition of a locally compact group $G$ looks as follows:
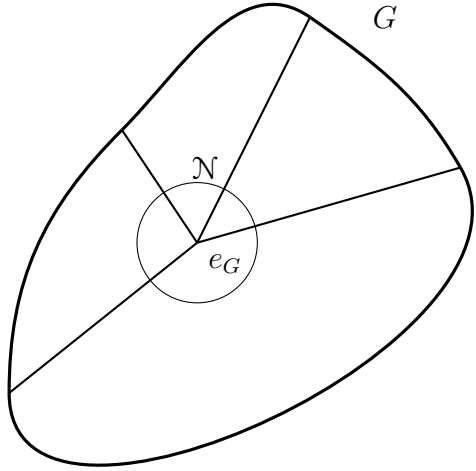
\begin{center}
\begin{tikzpicture}
\draw[thick] (2.5,2) node [below right]{$e_G$}--(1.5,3.5);
\draw (2.6,3.0) node {$\mathcal{N}$};
\draw [thick](2.5,2)--(4,5);
\draw [thick](2.5,2)--(6,3);
\draw [thick](2.5,2)--(0,0);
\draw (5,5) node {$G$};
\draw (2.5, 2) circle [radius=0.8];
\draw[very thick] (0,0) to [out=90,in=-135] (1.5,3.5);
\draw[very thick] (1.5, 3.5) to [out=45, in=145] (4,5);
\draw[very thick] (4, 5) to [out=-35, in=120] (6,3);
\draw[very thick] (6,3) to [out=-60, in=-90] (0,0);
\end{tikzpicture}
\end{center}

\subsection{d-torus}\

The {\bf Fej\'er kernel} for $\mathbb{T}$ is given by
$$K_n(t)=\displaystyle\sum_{j=-n}^n (1-\frac{|j|}{n+1})e(jt)=\frac{\sin^2{(n+1)\pi t}}{(n+1)\sin^2 \pi t}$$
 for every nonnegative integer $n$.

The Fej\'er kernel has many nice properties. Below we list some of them which would be frequently used throughout the paper. See~\cite[p.181, Prop. 3.1.10]{Grafakos2014} and~\cite[p.205, (3.4.3)]{Grafakos2014} for a proof.


\begin{proposition}~[Properties of $K_n$]~\label{pcfejer}
\begin{enumerate}
\item $K_n(t)\geq 0$ for every $t$ in $\mathbb{T}$ and $n\geq 0$.
\item $\int_{\mathbb{T}} K_n(t)\, dt=1$ for all $n\geq 0$.
\item For $\lambda\in(0,\frac{1}{2})$, we have $\displaystyle\sup_{t\in [\lambda,1-\lambda]} {K_n(t)}\leq\frac{1}{(n+1)\sin^2(\pi\lambda)},$   hence $\displaystyle\lim_{n\to \infty}\int_{\lambda}^{1-\lambda} K_n(t)\,dt=0.$
\end{enumerate}
\end{proposition}

So $\{K_n(t)\}_{n=1}^\infty$ is an approximate identity of  $L^1(\mathbb{T})$. In addition, $\displaystyle\lim_{n\to\infty}\sup_{t\in\mathcal{N}^c} |K_n(t)|=0$ for every neighborhood $\mathcal{N}$ of 0.

In the unit circle, there is some ''strange-looking'' local partition. For instance, choose two sequences $\{a_n\}_{n=1}^\infty$ and $\{b_n\}_{n=1}^\infty$ such that
\begin{itemize}
\item $0<a_{n+1}<b_{n+1}<a_n<b_n<\cdots<a_1<b_1=\frac{1}{2}$;
\item $\displaystyle \lim_{n\to\infty}a_n=\lim_{n\to\infty} b_n=0$.
\end{itemize}
Let $A_1=(\frac{1}{2},1)$, $\displaystyle A_2=\bigcup_{n=1}^\infty (a_n, b_n)$ and $\displaystyle A_3=\bigcup_{n=1}^\infty (b_{n+1}, a_n)$. Then $\{A_1, A_2, A_3\}$ is a local partition which is completely different from the local partition $\{(0,\frac{1}{2}), (\frac{1}{2},1)\}$ used in the classical Fej\'er's theorem.

So it is worthy of mentioning the following generalized Fej\'er's theorem for the unit circle.
\begin{corollary}
Given a local partition $\{A_1, A_2, \cdots, A_k\}$ of $\mathbb{T}$. Assume that $\displaystyle\lim_{n\to\infty} \int_{A_j} K_n(x)\,dx=\lambda_j$ for every $1\leq j\leq k$. For $f$ in $L^1(\mathbb{T})$, if each $f(x_0, A_j)$ exists at some $x_0$ in $\mathbb{T}$, then
$$\displaystyle\lim_{n\to\infty} K_n*f(x_0)=\sum_{j=1}^k \lambda_j f(x_0, A_j).$$
\end{corollary}

Now consider d-torus for $d\geq 2$.

We identify $\mathbb{T}^d$ with $[0,1)^d$ or $\mathbb{R}^d/\mathbb{Z}^d$ when necessary.  The inner product  $x\cdot y$ of $x=(x_1,\cdots,x_d)$ and $y=(y_1,\cdots,y_d)$ in $\mathbb{R}^d$ is given by $x_1y_1+\cdots+x_dy_d$.

The {\bf square Fej\'er  kernel} for $\mathbb{T}^d$ is defined by
$$K_n^d(x_1,\cdots,x_d)=\prod_{j=1}^d K_n(x_j)=\sum_{m\in\mathbb{Z}^d, |m_j|\leq n}(1-\frac{|m_1|}{n+1})\cdots (1-\frac{|m_d|}{n+1})e(m\cdot x)$$ for every nonnegative integer $n$.


\begin{proposition}~[Properties of $K_n^d$]~\label{pfejer}
\begin{enumerate}
\item $K_n^d(x)\geq 0$ for every $x$ in $\mathbb{T}^d$ and $n\geq 0$.
\item $\int_{\mathbb{T}^d} K_n^d(x)\, dx=1$ for all $n\geq 0$.
\end{enumerate}
\end{proposition}

\begin{proof}
All these properties of $K_n^d$ are induced by properties of $K_n$. See~\cite[p.181, Prop. 3.1.10 \& p.205, (3.4.3)]{Grafakos2014}.
\begin{enumerate}
\item $K_n^d\geq 0$ since $K_n\geq 0$.
\item $\int_{\mathbb{T}^d} K_n^d(x)\,dx=\prod_{j=1}^d\int_\mathbb{T} K_n(x_j)\,dx_j=1$.
\end{enumerate}
\end{proof}

Also $\{K_n^d\}_{n=1}^\infty$ is an approximate identity of $L^1(\mathbb{T}^d)$.

For $k=(k_1,\cdots,k_d)$ in $\{0,1\}^d$, define
$$I_k=\displaystyle\prod_{j=1}^d I_{k_j}$$ with $I_0=[0,\frac{1}{2})$ and $I_1=[\frac{1}{2},1)$. Note that $\int_0^{\frac{1}{2}}K_n(t)\,dt=\int_{\frac{1}{2}}^1 K_n(t)\,dt=\frac{1}{2}$ for all $n\geq 0$. So $\{I_k\}_{k\in\{0,1\}^d}$ is a local partition of $\mathbb{T}^d$~($=[0,1)^d$) such that
$$\int_{I_k}K_n^d(x)\,dx=\displaystyle\prod_{j=1}^d \int_{I_{k_j}} K_n(x_j)\,dx_j=\frac{1}{2^d}$$ for all $k\in\{0,1\}^d$ and $n\geq 0$.

For $f$ in $L^1(\mathbb{T}^d)$, define
$$f(x,I_k)=\displaystyle\lim_{\substack{y\to 0 \\ y\in I_k}} f(x-y)$$  for every $k\in\{0,1\}^d$.


\begin{corollary}~\label{cor:td}
Let $d\geq 2$. For  $f$ in $L^\infty(\mathbb{T}^d)$ and $x$ in $\mathbb{T}^d$, if each $f(x,I_k)$ exists, then
$$\displaystyle\lim_{n\to\infty}K_n^d*f(x)=\frac{1}{2^d}\displaystyle\sum_{k\in\{0,1\}^d} f(x,I_k).$$
\end{corollary}
\begin{proof}
Apply the first part of Theorem~\ref{Thm: GFejer} to  the approximate identity $\{K_n^d\}_{n=0}^\infty$ of $L^1(\mathbb{T}^d)$ and the local partition $\{I_k\}_{k\in\{0,1\}^d}$ of $\mathbb{T}^d$.
\end{proof}

\subsection{Euclidean spaces}\

The {\bf Wigner semicircle kernel} $W_\theta(t)$ is given by
\[ W_\theta= \left\{
  \begin{array}{l l}
    \frac{2}{\pi\theta^2}\sqrt{\theta^2-t^2} & \quad \text{when $-\theta\leq t\leq\theta,$}  \\
    0 & \quad \text{otherwise}
  \end{array} \right.\]
  for all $t\in\mathbb{R}$ and $\theta>0$.

For a $\lambda$ in $(0,1)$, by shifting the graph of $W_\theta(t)$ along the $t$-axis, we get a new function $W_{\theta,\lambda}(t)$ satisfying
$$\int_{(-\infty,0)}W_{\theta,\lambda}(t)\,dt=\lambda.$$ Also note that every $W_{\theta,\lambda}(t)$ is compactly supported in a closed interval  shrinking to $\{0\}$ as $\theta$ goes to 0.

Define $W_{\theta,\lambda}^d(x)=\displaystyle\prod_{j=1}^d W_{\theta,\lambda}(x_j)$ for any positive integer $d$, then $\{W_{\theta,\lambda}^d(x)\}_{\theta>0}$ is an approximate identity of $L^1(\mathbb{R}^d)$ and satisfies that $\displaystyle\lim_{\theta\to 0}\sup_{x\in\mathcal{N}^c} |W_{\theta,\lambda}^d(x)|=0$ for every neighborhood
 $\mathcal{N}$ of 0 in $\mathbb{R}^d$.

\begin{corollary}~\label{cor: wk}
Fix a $\lambda$ in $(0,1)$. For an $f$ in  $L^1(\mathbb{R}^d)$ or $L^\infty(\mathbb{R}^d)$, if every $f(x,J_k)$ exists for some $x\in\mathbb{R}^d$, then
$$\displaystyle\lim_{\theta\to 0}W_{\theta,\lambda}^d*f(x)=\displaystyle\sum_{k\in\{0,1\}^d} \prod_{j=1}^d \lambda^{1-k_j}(1-\lambda)^{k_j}f(x,J_k).$$
\end{corollary}
\begin{proof}
Apply Theorem~\ref{Thm: Gfejer} to Consider the local partition $\{J_k\}_{k\in\{0,1\}^d}$ of $\mathbb{R}^d$ and the approximate identity $\{W_{\theta,\lambda}^d\}_{\theta>0}$ of $L^1(\mathbb{R}^d)$. Note that $\displaystyle\lim_{\theta\to 0}\sup_{x\in\mathcal{N}^c} |W_{\theta,\lambda}^d(x)|=0$ for every neighborhood
 $\mathcal{N}$ of 0 in $\mathbb{R}^d$. Furthermore
 $$\int_{J_k}W_{\theta,\lambda}^d(x)\,dx=\displaystyle\prod_{l=1}^d \int_{J_{k_l}} W_{\theta,\lambda}(x_l)\,dx_l=\prod_{j=1}^d \lambda^{1-k_j}(1-\lambda)^{k_j}.$$
Applying Theorem~\ref{Thm: GFejer} finishes the proof.
\end{proof}


\begin{remark}

In~\cite{FeichtingerWeisz2006}, H. G. Feichtinger and F. Weisz prove some theorems about convergence for convolutions of some special types of approximate identities with $f$ in $L^1(\mathbb{R}^d)$ or $L^1(\mathbb{T}^d)$ at a Lebesgue point of $f$. Cf.~\cite[Thm. 4.6 \& Thm. 7.2]{FeichtingerWeisz2006}). Corollaries~\ref{cor:td} and ~\ref{cor: wk} have some overlaps with, but are not covered by those in~\cite{FeichtingerWeisz2006} since the points satisfying the assumptions of Corollaries~\ref{cor:td} and ~\ref{cor: wk}  are  not necessarily Lebesgue points.

A {\bf Lebesgue point} of an $f$ in $L^1(\mathbb{R}^d)$ is a point $x$ in $\mathbb{R}^d$ such that
$$\displaystyle\lim_{r\to 0}\frac{1}{m(B_r(x))}\int_{B_r(x)} |f(y)-f(x)|dm(y)=0.$$

Here $m$ is the Lebesgue measure of $\mathbb{R}^d$ and $B_r(x)$ is the open ball centered at $x$ with radius $r>0$~\cite[7.6]{Rudin1987}.

A point $x\in\mathbb{R}^d$ so that every  $f(x,J_k)$ exists is not necessarily a Lebesgue point.

For instance, consider $f={\bf 1}_{(0,1)}$, the characteristic function of $(0,1)$, which is in $L^1(\mathbb{R})$. Then $f(0-)=0$ and $f(0+)=1$.
 However 0 is not a Lebesgue point~\footnote{We can change the value of  $f(0)$, but this does not affect the fact that 0 is not a Lebesgue point.} since

$$\displaystyle\lim_{r\to 0}\frac{1}{2r}\int_{(-r,r)} |f(y)-f(0)|dy=\lim_{r\to 0}\frac{1}{2r}\int_{(0,r)} |1-0|dy=\frac{1}{2}.$$
\end{remark}

\subsection{The Heisenberg group}\

The continuous {\bf Heisenberg group} $\mathbb{H}$ is the group of 3 by 3 upper triangular real matrices with diagonal entries 1, that is,
$$\mathbb{H}=\{\bigl(\begin{smallmatrix}
1&a&b \\ 0&1&c\\ 0&0&1
\end{smallmatrix} \bigr)| a,b,c\in\mathbb{R}\}.$$
As a topological space $\mathbb{H}$ is homeomorphic to $\mathbb{R}^3$ and a left Haar measure of $\mathbb{H}$ is $da\,db\,dc$.

Define $$W_\theta^3(a,b,c)=W_\theta(a)W_\theta(b)W_\theta(c)$$ for every $\theta>0$ and $a,b,c$ in $\mathbb{R}$. Then it is easy to see that $\{W_\theta^3\}_{\theta>0}$ is an approximate identity of $L^1(\mathbb{H})$ and $$\displaystyle\lim_{\theta\to 0}\sup_{(a,b,c)\in\mathcal{N}^c} |W_{\theta}^3(a,b,c)|=0$$ for every neighborhood
$\mathcal{N}$ of $e_\mathbb{H}=\bigl(\begin{smallmatrix}
1&0&0 \\ 0&1&0\\ 0&0&1
\end{smallmatrix} \bigr)$ in $\mathbb{H}$.

For $k=(k_1,k_2,k_3)$ in $\{0,1\}^3$, define
$J_k=\displaystyle\prod_{l=1}^3 J_{k_l}$ with $J_0=(-\infty,0)$ and $J_1=[0,\infty)$.  Then $\{J_k\}_{k\in\{0,1\}^3}$ is a local partition of $\mathbb{H}$ at $e_\mathbb{H}$ such that $\int_{J_k} W_\theta^3(a,b,c) dadbdc=\frac{1}{8}$ for every $k\in\{0,1\}^3$.

We get the following.


\begin{corollary}~\label{cor:H}
For an $f$ in  $L^1(\mathbb{H})$ or $L^\infty(\mathbb{H})$, if every $f(x,J_k)$ exists for some $x$ in $\mathbb{H}$, then
$$\displaystyle\lim_{\theta\to 0}W_\theta^3*f(x)=\displaystyle\frac{1}{8}\sum_{k\in\{0,1\}^3} f(x,J_k).$$
\end{corollary}
\begin{proof}
Applying Theorem~\ref{Thm: GFejer} to the local partition $\{J_k\}_{k\in\{0,1\}^3}$  of $\mathbb{H}$ gives the proof.
\end{proof}

\end{document}